\documentclass{amsart}

\usepackage{ulsy}
\usepackage{amsmath}
\usepackage{algorithm}
\usepackage{algpseudocode}
\usepackage{amsfonts}
\usepackage{amssymb}
\usepackage{amsthm}
\usepackage{color}

\newcommand{\Z}{\mathbb{Z}}

\newcommand{\Q}{\mathbb{Q}}

\newcommand{\bp}{\begin{problem}}

\newcommand{\ep}{\end{problem}}

\newcommand{\ba}{\begin{answer}}

\newcommand{\ea}{\end{answer}}

\newcommand{\ben}{\renewcommand{\theenumi}{\alph{enumi}}

\renewcommand{\labelenumi}{(\theenumi)}\begin{enumerate}}

\newcommand{\een}{\end{enumerate}}

\newcommand{\Aut}{\mathrm{Aut}}

\newtheorem{defin}{Definition}[section]

\newtheorem{lem}[defin]{Lemma}

\title[A Polynomial Time Algorithm For the Conjugacy Problem in $\Z^n \rtimes \Z$]{A Polynomial Time Algorithm For The Conjugacy Decision and Search Problems in Free Abelian-by-Infinite Cyclic Groups}

\begin{document}

\date{August 2, 2014}

\author{Bren Cavallo}
\address{Bren Cavallo, Department of Mathematics, CUNY Graduate Center, City University of New York}
\email{bcavallo@gc.cuny.edu}

\author{Delaram Kahrobaei}
\address{Delaram Kahrobaei, CUNY Graduate Center, PhD Program in Computer Science and NYCCT, Mathematics Department, City
University of New York} \email{dkahrobaei@gc.cuny.edu}
\nocite{*}
\begin{abstract}
In this paper we introduce a polynomial time algorithm that solves both the conjugacy decision and search problems in free abelian-by-infinite cyclic groups where the input is elements in normal form.  We do this by adapting the work of Bogopolski, Martino, Maslakova, and Ventura in \cite{bogopolski2006conjugacy} and Bogopolski, Martino, and Ventura in \cite{bogopolski2010orbit}, to free abelian-by-infinite cyclic groups, and in certain cases apply a polynomial time algorithm for the orbit problem over $\Z^n$ by Kannan and Lipton \cite{kannan1986polynomial}.
\end{abstract}

\maketitle

\thispagestyle{empty}

\section{Introduction}
The conjugacy decision problem in a finitely presented group $G$, is determining if there is a solution to the equation $v = x u x^{-1}$ where $u,v,x \in G$.  The decision problem also has the search variant, given $u$ and $v$ conjugate, find an explicit $x$ that conjugates $u$ to $v$.   The conjugacy decision problem is in general undecidable \cite{miller1992decision}, whereas the search problem is decidable in every recursively presented group \cite{groupcrypto}.

\

Due to the rise of applications of group theory to computer science and cryptography, more research has been directed towards studying the algorithmic complexity of group theoretic algorithms rather than solely investigating decidability.  Other polynomial time algorithms for the conjugacy problem in solvable groups are due to Vassileva in free solvable groups \cite{vassileva2011polynomial} and Diekert, Miasnikov, and Wei{\ss} in solvable Baumslag-Solitar groups \cite{diekert2013conjugacy}.  Some very related results can also be seen in the work of Sale \cite{sale2011short, sale2012conjugacy}, in which he shows that for a special class of the groups studied in this paper, the conjugacy length function is bounded from above by a linear function.  Namely, for any two conjugate elements in these groups, there exists a conjugator of geodesic length less than a constant multiple of the sum of the geodesic lengths of the elements.  

\ 

In the following sections we introduce a polynomial time algorithm that solves both the conjugacy decision and search problems in free abelian-by-infinite cyclic groups where elements are given in terms of their normal forms. This family of groups is polycyclic so it is well known that they have a solvable conjugacy problem.  This fact is due originally to Formanek \cite{formanek1976conjugate} and Remesslennikov \cite{remeslennikov1969conjugacy} who independently proved that virtually polycyclic groups are conjugacy separable:  for any two $u, v \in G$ that are not conjugate, there exists a finite homomorphic image in which the images of $u$ and $v$ are not conjugate.  Conjugacy can be solved in such groups by conjugating $u$ by elements of $G$ and checking if the result is $v$ while simultaneously enumerating all homomorphisms from $G$ into a finite group and checking if the images of $u$ and $v$ are conjugate.  One of the processes is guaranteed to stop which then provides an answer to the problem.  This algorithm is brute force and clearly may take very long even in simple cases.

\

We start the paper with a review of free abelian-by-infinite cyclic groups and the twisted conjugacy problem. We then detail the algorithm due to Bogopolski, Martino, Maslakova, and Ventura from \cite{bogopolski2006conjugacy} and prove that it, along with the solution to the orbit problem due to Kannan and Lipton \cite{kannan1986polynomial}, solves both the conjugacy decision and search problems in polynomially many steps with respect to the lengths of the inputs in normal form.  Finally we end with a complexity analysis of the algorithm and discuss how the complexity changes when inputs are considered in their geodesic forms rather than normal forms.

\section{Free Abelian-by-Infinite Cyclic Groups}
We say that a group $G$ is \emph{free abelian-by-cyclic} if $G$ fits into a short exact sequence of the form:

$$1 \rightarrow \Z^n \rightarrow G \rightarrow C \rightarrow 1$$

where $C$ is a cyclic group.  If $C \simeq \Z$, then we say $G$ is \emph{free abelian-by-infinite cyclic}.  In this case, $G$ splits as $\Z^n \rtimes_\phi \Z$ for some $\phi \in \mbox{GL}_n (\Z)$.  Therefore, $G$ has the presentation:

$$\langle g_1, g_2, \cdots, g_n, t \, | \, t g_i t^{-1} = \phi (g_i), [g_i, g_j] = 1 \rangle$$

where $1 \leq i < j \leq n$ and where we view the $g_i$ as the generators of $\Z^n$ and $t$ as the generator of $\Z$.  As such, any $g \in G$ can be written as $w_1 t^{k_1} w_2 t^{k_2} \cdots w_m t^{k_m}$ where each $w_i \in \Z^n$ and $k_i \in \Z$.  Applying the relations of the form $t g_i t^{-1} = \phi (g_i)$ multiple times, one can move all the $t^{k_i}$ over to the right side of the word, thus representing each element as $wt^k$ where $w \in \Z^n$ and $t^k \in \Z$.  For any $g \in G$ we call such a representative its \emph{normal form}.  Multiplication in normal forms can then be carried out as:

$$w t^k \cdot w' t^{k'} = w\phi^k(w')t^{k + k'}.$$

Namely, every time we need to move $t^k$ to the right, over a word in $\Z^n$, we can do so by applying $\phi^k$.  It can additionally be seen (see \cite{eick2001algorithms}) that each group element's normal form is unique.

\

For the remainder of this paper, we will be working entirely with elements in their normal forms and as such assume in the following algorithm that elements are given in their normal form.  We also define a length function, $|\cdot|$, over elements of $G$ where if $g =_G w t^k$, then:
$$|g| = |w t^k| = |w|_{\Z^n} + |k|$$
where $|w|_{\Z^n}$ is the standard geodesic length of $w \in \Z^n$.

\section{The Twisted Conjugacy Problem}
\begin{defin}
Given a finitely presented group, $G$, an autormorphism $\phi \in \Aut(G)$, and $u, v \in G$ we say $u$ and $v$ are \emph{twisted conjugate} by $\phi$ if there exists $x \in G$ such that
$$v = x u \phi(x^{-1}).$$

If $u$ and $v$ are twisted conjugate by $\phi$ we write:
$$u \sim_\phi v.$$
\end{defin}

Notice that the standard conjugacy problem is a special case of the twisted conjugacy problem by taking $\phi$ to be the identity.

\

In \cite{bogopolski2006conjugacy} Bogopolski, Martino, Maslakova, and Ventura  introduced an algorithm that relates the conjugacy problem in free-by-infinite cyclic groups to the twisted conjugacy problem in free groups.   Following that work, Bogopolski, Martino, and Ventura \cite{bogopolski2010orbit} adapted the algorithm from \cite{bogopolski2006conjugacy} to solve the conjugacy problem in a variety of groups created by extensions.   What follows is an adaptation of their algorithm for free abelian-by-cyclic groups.

\section{The Algorithm}
The following lemma and proof is taken directly from the beginning of section 2 in \cite{bogopolski2006conjugacy} and adapted to free abelian-by-infinite cyclic groups.
\begin{lem} Let $u = w t^s$ and $v =x  t^r $ in $\Z^n \rtimes_\phi \Z$ be conjugate.  Then $s = r$ and there exists $e \in \Z$ such that $\phi^{e}(w) \sim_{\phi^s} x$ in $\Z^n$.  Additionally, if $\phi^s = \phi^r$ is the identity, then $x = \phi^e (w)$ for some $e \in \Z$.
\end{lem}
\begin{proof}
Let $a =  b t^e\in \Z^n \rtimes_\phi \Z$ such that $v = a u a^{-1}$.  Therefore
$$x t^r = (b t^e )  w t^s (b t^e)^{-1} = b t^e  w t^s   t^{-e} b^{-1} =  $$
$$ b  \phi^e (w) t^s b^{-1} = b  \phi^{e} (w) \phi^{s} (b^{-1}) t^s.$$ 
As such, we have:
$$x t^r = b  \phi^{e} (w) \phi^{s} (b^{-1}) t^s$$
which implies that $s = r$, and that $\phi^{e} (w) \sim_{\phi^s} x$ by $b$.  
\end{proof}

Given $u$ and $v$ as above, the lemma shows that there are two cases one must consider to solve the conjugacy decision and search problems in $\Z^n$-by-$\Z$ groups.  First check if $s = r$.  If not, then $u$ and $v$ are not conjugate.  If the exponents are the same, then there are two cases:

\

\begin{itemize}
\item If $\phi^s$ is trivial, we have to decide if $\exists e \in \Z$ such that $x = \phi^e (w)$.
\item Otherwise, we have to decide if there exists $e$ such that $\phi^{e} (w) \sim_{\phi^s} x$.
\end{itemize}

\

The first case, namely given two vectors $w, x \in \Z^n$ and $\phi \in \mbox{GL}_n (\Z)$ determine if there exists $e \in \Z$ such that $x =  \phi^e (w)$, is known as the orbit problem over $\Z^n$.  In \cite{kannan1986polynomial} Kannan and Lipton provide a polynomial time algorithm that solves the orbit problem over $\Q^n$.  Since the orbit problem over $\Z^n$ is a special case of their work, this algorithm provides a polynomial time solution to the twisted conjugacy problem over $\Z^n$ in the case that $\phi^s$ is trivial.  If such an $e$ is found that satisfies the orbit problem, then we have that $v = t^e u t^{-e}$.

\

For the second case, we use the fact from the lemma that $\exists b \in \Z^n$, $e \in \Z$ such that $x = b \phi^{e} (w) \phi^{s} (b^{-1})$.   Before we begin the algorithm, we state {\bf Lemma 1.7} from \cite{bogopolski2006conjugacy}.

\begin{lem}
For any group $G$, $\phi \in \Aut(G)$, and $u \in G$, $u \sim_\phi \phi(u)$.
\end{lem}
\begin{proof}
$\phi (u) =  u^{-1} u \phi(u)$.  Therefore $u$ is twisted conjugate over $\phi$ to $\phi(u)$ by $u^{-1}$.
\end{proof}

As such, $\phi^{e} (w) \sim_{\phi^s}  \phi^{e \pm ks}(w)$ for any $k \in \Z$.  Therefore, if there exists an $e$ that satisfies the equation $\phi^e(w) \sim_{\phi^s} x$, then we can find such an $e$ among $\{0, 1, \cdots, |s| - 1\}$.  This is where it is important that we are in the second case as $s$ is not zero.

\

We can now proceed with the full algorithm.   Due to the fact that $x, w \in \Z^n$ and $\phi \in \mbox{GL}_n (\Z)$  it is more convenient to put the equation $x = b \phi^e (w)  \phi^s (b^{-1})$ into additive notation.  As such we write,
$$x = b + \phi^e (w) -  \phi^s (b).$$
 This gives the equation
$$x - \phi^e(w) = (\mbox{Id}_n  - \phi^s)b$$

 where $\mbox{Id}_n$ is the  $n \times n$ identity matrix.  In this way, each $e$ yields a system of linear equations in which we solve for the vector $b$.  There will be a solution to the conjugacy problem, as long as there is some $e$ for which the solution $b$ is in $\Z^n$.  Moreover, we know that if there is a solution to the conjugacy problem, such an $e$ must lie in the set $\{0, 1, \cdots , |s| - 1\}$.  If there exists such an $e$, $u \sim v$ and $bt^e$ is a conjugator.   As such, we proceed by solving the system of linear equations given by each of the possible $e$'s and then checking if the solution, $b$ is in $\Z^n$. In the case that $\mbox{Id}_n  - \phi^s$ is invertible, namely, $\phi^s$ does not have 1 as an eigenvalue, then we can also write:

$$b = (\mbox{Id}_n  - \phi^s)^{-1} (x - \phi^e(w))$$

 For a complete description of the algorithm in pseudocode on inputs $w t^s, x t^r \in \Z^n \rtimes_\phi \Z$, see {\bf{Algorithm 1}} on the following page.  We have the algorithm return {\bf FALSE} if the elements are not conjugate, and a conjugating element if they are.

\begin{algorithm}
\caption{Conjugacy Algorithm for $\Z^n \rtimes_\phi \Z$}
\begin{algorithmic}
\If {$s \neq r$}
    \State return {\bf FALSE}
\ElsIf{$\phi^s$ is the identity}
    \State Run Kannan-Lipton algorithm.
        \If {Kannan-Lipton returns $k$}
                \State return $t^k$
                \Else \,  return {\bf FALSE}
        \EndIf
\Else
         \State{$e := 0$}
    \While {$e<|s|$}
        \If{$\exists b \in \Z^n$ such that $x - \phi^e(w) = (\mbox{Id}_n  - \phi^s)b$}
                \State return $b t^e$
        \Else $\, e := e + 1$
        \EndIf
     \EndWhile
     \State return {\bf FALSE}
    \EndIf

\end{algorithmic}
\end{algorithm}

\section{Complexity Analysis}
In the algorithm above we have two cases each of which can dealt with in polynomially many steps with respect to $n$ and $|s|$.  If $s = r \neq 0$, we find solutions of an $n \times n$ linear system at most $|s|$ times.  On the other hand, if $|s| = |r| = 0$, we use the algorithm of Kannan and Lipton which runs in polynomial time.  Therefore, this algorithm is at most polynomial in terms of $n$ and the lengths of the input words.

\

It is worth pointing out that unlike many of the algorithms group theorists study, this algorithm takes as inputs words in their polycyclic normal forms as opposed to in their geodesic form or just in any general form.  This affects the complexity of the algorithm as all forms have different lengths.  It is worth noting that the geodesic form of a word in a polycyclic group can be logarithmic with respect to the length in normal form.  For instance in the group:

$$G = \Z^2 \rtimes_\phi \Z = \langle g_1, g_2, t \, | \, [g_1, g_2], tg_1t^{-1} = g_1^2 g_2, t g_2 t^{-1} = g_1 g_2 \rangle$$

where $\phi (t) = \begin{pmatrix}
2 & 1\\ 
1 & 1
\end{pmatrix}$ , we have that:

$$t^n  a b t^{-n} = a^{F(2n + 2)} b^{F(2n + 1)}$$

where $F(n)$ is the $n^{th}$ element of the Fibonacci sequence $F = \{1,1,2,3,5, \cdots \}$.  In this way, normal forms in $G$ can be exponentially longer than their geodesic forms.  As such, collecting words in geodesic form and then performing the algorithm would take an exponential number of steps with respect to the geodesic length since the process of collecting involves writing out a word that is exponentially longer than the original word.  On the other hand, in a practical setting, converting words to normal forms is fast (see \cite{eick2004polycyclic}) and the main complexity involved in the algorithm has to do with the exponent above the generator $t$ after collection, which is just the sum of the exponents above the $t$'s in a general word.  As such, after collection, the exponent above $t$ contributes to the length of the word at most what it contributed prior to collection.  In that vein, even though a word may grow in size exponentially after collection, most of the additional steps are involved in collection rather than in actually solving the conjugacy problem.

\section{Acknowledgments and Support}
Delaram Kahrobaei is partially supported by the Office of Naval Research grant N00014120758, the American Association for the Advancement of Science, a PSC-CUNY grant from the CUNY research foundation, as well as the City Tech foundation.

We would also like to thank the Universitat Politècnica de Catalunya where much of this research was conducted and our host Enric Ventura.  Finally, we would like to thank the reviewer(s) for their comments.

\bibliographystyle{plain}
\bibliography{conjbib}

\begin{thebibliography}{10}

\bibitem{bogopolski2006conjugacy}
O~Bogopolski, A~Martino, O~Maslakova, and E~Ventura.
\newblock The conjugacy problem is solvable in free-by-cyclic groups.
\newblock {\em Bulletin of the London Mathematical Society}, 38(05):787--794,
  2006.

\bibitem{bogopolski2010orbit}
Oleg Bogopolski, Armando Martino, and Enric Ventura.
\newblock Orbit decidability and the conjugacy problem for some extensions of
  groups.
\newblock {\em Transactions of the American Mathematical Society},
  362(4):2003--2036, 2010.

\bibitem{collins1977conjugacy}
Donald~J Collins and Charles~F Miller.
\newblock The conjugacy problem and subgroups of finite index.
\newblock {\em Proceedings of the London Mathematical Society}, 3(3):535--556,
  1977.

\bibitem{diekert2013conjugacy}
Volker Diekert, Alexei Miasnikov, and Armin Wei{\ss}.
\newblock Conjugacy in baumslag's group, generic case complexity, and division
  in power circuits.
\newblock {\em arXiv preprint arXiv:1309.5314}, 2013.

\bibitem{drutu2011lectures}
Cornelia Drutu and Michael Kapovich.
\newblock Lectures on geometric group theory.
\newblock {\em preprint}, 2013.

\bibitem{eick2001algorithms}
B.~Eick.
\newblock {\em Algorithms for Polycyclic Groups}.
\newblock Habilitationsschrift, Universitat Kassel, 2001.

\bibitem{eick2004polycyclic}
Bettina Eick and Delaram Kahrobaei.
\newblock Polycyclic groups: A new platform for cryptology?
\newblock {\em arXiv preprint math/0411077}, 2004.

\bibitem{formanek1976conjugate}
Edward Formanek.
\newblock Conjugate separability in polycyclic groups.
\newblock {\em Journal of Algebra}, 42(1):1--10, 1976.

\bibitem{handbook}
Derek~F. Holt, Bettina Eick, and Eamonn~A. O'Brien.
\newblock {\em Handbook of Computational Group Theory}.
\newblock CRC Press, 2005.

\bibitem{kannan1986polynomial}
Ravindran Kannan and Richard~J Lipton.
\newblock Polynomial-time algorithm for the orbit problem.
\newblock {\em Journal of the ACM (JACM)}, 33(4):808--821, 1986.

\bibitem{miller1992decision}
Charles~F Miller~III.
\newblock Decision problems for groups—survey and reflections.
\newblock In {\em Algorithms and classification in combinatorial group theory},
  pages 1--59. Springer, 1992.

\bibitem{groupcrypto}
Alexei Myasnikov, Vladimir Shpilrain, and Alexander Ushakov.
\newblock {\em Group-based Cryptography}.
\newblock Springer, 2008.

\bibitem{remeslennikov1969conjugacy}
Vladimir~N Remeslennikov.
\newblock Conjugacy in polycyclic groups.
\newblock {\em Algebra and Logic}, 8(6):404--411, 1969.

\bibitem{sale2011short}
Andrew~W Sale.
\newblock Short conjugators in solvable groups.
\newblock {\em arXiv preprint arXiv:1112.2721}, 2011.

\bibitem{sale2012conjugacy}
Andrew~W Sale.
\newblock Conjugacy length in group extensions.
\newblock {\em arXiv preprint arXiv:1211.3144}, 2012.

\bibitem{vassileva2011polynomial}
Svetla Vassileva.
\newblock Polynomial time conjugacy in wreath products and free solvable
  groups.
\newblock {\em Groups, Complexity, Cryptology}, 3(1):105--120, 2011.

\end{thebibliography}

\end{document}